\documentclass[12pt]{article}

\usepackage{graphicx}
\usepackage{nameref}
\usepackage[shortlabels,inline]{enumitem}
\usepackage{empheq}
\usepackage[shortlabels]{enumitem}
\setlist[enumerate]{nosep}
\usepackage[doc]{optional}
\usepackage{xcolor}
\usepackage[colorlinks=true,
            linkcolor=refkey,
            urlcolor=lblue,
            citecolor=red]{hyperref}
\usepackage{float}
\usepackage{soul}
\usepackage{graphicx}
\definecolor{labelkey}{rgb}{0,0.08,0.45}
\definecolor{refkey}{rgb}{0,0.6,0.0}
\definecolor{Brown}{rgb}{0.45,0.0,0.05}
\definecolor{lime}{rgb}{0.00,0.8,0.0}
\definecolor{lblue}{rgb}{0.5,0.5,0.99}

 \usepackage{mathpazo}

\colorlet{hlcyan}{cyan!30}

\usepackage{stmaryrd}
\usepackage{amssymb}
\makeatletter
\def\namedlabel#1#2{\begingroup
   \def\@currentlabel{#2}%
   \label{#1}\endgroup
}
\makeatother

\newcommand{\sepp}{\setlength{\itemsep}{-2pt}}

\newcommand{\moyo}[2]{\ensuremath{\sideset{^{#2}}{}{\operatorname{}}\!#1}}

\newcommand{\thalb}{\ensuremath{\tfrac{1}{2}}}
\newcommand{\menge}[2]{\big\{{#1}~\big |~{#2}\big\}}

\newcommand{\fenv}[1]%
{\ensuremath{\,\overrightarrow{\operatorname{env}}_{#1}}}
\newcommand{\benv}[1]%
{\ensuremath{\,\overleftarrow{\operatorname{env}}_{#1}}}

\newcommand{\scal}[2]{\left\langle{#1},{#2}  \right\rangle}

\newcommand{\RR}{\ensuremath{\mathbb R}}

\newcommand{\dom}{\ensuremath{\operatorname{dom}}}

\newcommand{\ran}{\ensuremath{{\operatorname{ran}}\,}}

\newcommand{\conv}{\ensuremath{\operatorname{conv}\,}}

\newcommand{\cdom}{\ensuremath{\overline{\operatorname{dom}}\,}}
\newcommand{\Fix}{\ensuremath{\operatorname{Fix}}}
\newcommand{\Id}{\ensuremath{\operatorname{Id}}}

\newcommand{\pinf}{\ensuremath{+\infty}}

{\begin{list}{}{%
\settowidth{\labelwidth}{\textrm{#1~}}%
\setlength{\leftmargin}{\labelwidth+\labelsep}}}
{\end{list}}
\usepackage{amsthm}
\usepackage[capitalize,nameinlink]{cleveref}
\crefname{equation}{}{equations}
\crefname{chapter}{Appendix}{chapters}
\crefname{item}{}{items}
\crefname{enumi}{}{}

\theoremstyle{definition}
\newtheorem{theorem}{Theorem}[section]

\newtheorem{corollary}[theorem]{Corollary}

\newtheorem{proposition}[theorem]{Proposition}

\newtheorem{ex}[theorem]{Example}

\newtheorem{remark}[theorem]{Remark}

\providecommand{\RR}{\mathbb{R}}

\providecommand{\conv}{\operatorname{conv}}
\providecommand{\ran}{\operatorname{ran}}

\providecommand{\dom}{\operatorname{dom}}

\providecommand{\Id}{\operatorname{{ Id}}}

\providecommand{\ran}{\operatorname{ran}}

\providecommand{\Id}{\operatorname{Id}}

\newcommand{\cran}{\ensuremath{\overline{\operatorname{ran}}\,}}

\providecommand{\RR}{\mathbb{R}}

\definecolor{myblue}{rgb}{.8, .8, 1}

\allowdisplaybreaks 

\begin{document}

\title{\textsc{
Resolvents and Yosida approximations \\
of displacement mappings of isometries\footnote{
Dedicated to Terry Rockafellar on the occasion of 
his 85th birthday
}}}

\author{
Salihah Alwadani\thanks{
Mathematics, University
of British Columbia,
Kelowna, B.C.\ V1V~1V7, Canada. E-mail:
\texttt{saliha01@mail.ubc.ca}.},~ 
Heinz H.\ Bauschke\thanks{
Mathematics, University
of British Columbia,
Kelowna, B.C.\ V1V~1V7, Canada. E-mail:
\texttt{heinz.bauschke@ubc.ca}.},~
Julian P.\ Revalski\thanks{
Institute of Mathematics and Informatics,
Bulgarian Academy of Sciences, 
Acad.\ G.\ Bonchev str., Block~8, 
1113~Sofia, Bulgaria. E-mail: 
\texttt{revalski@math.bas.bg}.},
and
Xianfu Wang\thanks{
Mathematics, University
of British Columbia,
Kelowna, B.C.\ V1V~1V7, Canada. E-mail:
\texttt{shawn.wang@ubc.ca}.}
}

\date{March 23, 2021} 
\maketitle

\vskip 8mm

\begin{abstract} \noindent
Maximally monotone operators are fundamental objects in modern optimization. 
The main classes of monotone operators are subdifferential operators and 
matrices with a positive semidefinite symmetric part. 
In this paper, we study a nice class of monotone operators: 
displacement mappings of isometries of finite order. We derive explicit
formulas for resolvents, Yosida approximations, and (set-valued and Moore-Penrose)
inverses. 
We illustrate our results by considering certain rational rotators and 
circular shift operators. 
\end{abstract}

{
\noindent
{\bfseries 2020 Mathematics Subject Classification:}
{Primary 
47H05, 
47H09; 
Secondary 
47A06, 
90C25. 
}

\noindent {\bfseries Keywords:}
Circular shift, 
displacement mapping,
isometry of finite order, 
maximally monotone operator,
Moore-Penrose inverse, 
nonexpansive mapping,
resolvent,
set-valued inverse,
Yosida approximation. 
}

\section{Introduction}

Throughout this paper, we assume that 
\begin{empheq}[box=\fbox]{equation*}
    \text{$X$ is
    a real Hilbert space with inner product
    $\scal{\cdot}{\cdot}\colon X\times X\to\RR$, }
\end{empheq}
and induced norm $\|\cdot\|$, 
that $X\neq\{0\}$, and that 
$R\colon X\to X$ is a linear isometry of 
finite order $m\in\{2,3,\ldots\}$:
\begin{empheq}[box=\fbox]{equation*}
R^m=\Id.
\end{empheq}
Here $\Id$ denotes the identity operator on $X$.
It follows that $R$ is surjective and that 
\begin{equation*}
  \text{ $\|R\|=1$ and hence $R$ is nonexpansive,}
\end{equation*}
i.e., Lipschitz continuous with constant $1$. 
Therefore, 
by, e.g., \cite[Theorem~VI.5.1]{Berberian},
\begin{equation*}
  R^* = R^{-1} = R^{m-1}.
\end{equation*}
We also define throughout the paper 
\begin{empheq}[box=\fbox]{equation*}
M := \Id-R. 
\end{empheq}
Following \cite[Exercise~12.16]{Rock98},  
we shall refer to $M$ as the \emph{displacement mapping} of $R$.
Indeed, \cite[Exercise~12.16]{Rock98} states that $M$
is \emph{maximally monotone} when $X=\RR^n$; this result 
remains true in general as well \cite[Example~20.29]{BC2017}.
Monotone operators play a major role in modern optimization due to the fact 
that their zeros are often solutions to inclusion or optimization problems.
For more on monotone operator theory, we refer the reader to 
\cite{BC2017}, 
\cite{Brezis},
\cite{BurIus},
\cite{Rock98},
\cite{Simons1},
\cite{Simons2}, 
\cite{Zeidler2a},
and 
\cite{Zeidler2b}.
The main examples of monotone operators are subdifferential operators of convex functions 
and positive semidefinite matrices. 

Displacement mappings of nonexpansive mappings 
form a nice class of monotone operators. 
These have turned out to be very useful in optimization.
Here are some examples. 
The papers \cite{21} and \cite{Victoria} on asymptotic regularity of 
projection mappings and firmly nonexpansive mappings rely critically on the displacement mapping $M$. 
The analysis of the range of the Douglas--Rachford operator in \cite{94} employed
displacement mappings to obtain duality results. Asymptotic regularity results 
for nonexpansive mappings were generalized in \cite{111} to displacement mappings. 
In turn, a new application of the Brezis--Haraux theorem led to a completion of this 
study in the recent paper \cite{125} along with sharp and limiting examples. 

\emph{The purpose of this paper is to present a comprehensive analysis of $M$
from the point of view of monotone operator theory. 
We provide elegant and explicit formulas for resolvents and Yosida approximations for 
$M$ and its inverse.}

The paper is organized as follows.
In \cref{sec:Julian}, we derive a formula for the resolvent of $\gamma M$ 
(see \cref{l:Julian}) and discuss asymptotic behaviour as $\gamma\to 0^+$ or 
$\gamma\to\pinf$. The set-valued and Moore-Penrose inverses of $M$ 
are provided in \cref{sec:inverse}. 
In \cref{sec:ResYo}, we obtain formulae for the resolvents and 
Yosida approximations. \cref{sec:examples} presents concrete 
examples based on rational rotators and circular shift operators. 
The final \cref{sec:conc} offers some concluding remarks. 
Notation is standard and follows largely \cite{BC2017}.

\section{The resolvent of $\gamma M$}

\label{sec:Julian}

Associated with any maximally monotone operator $A$ 
is the so-called resolvent $J_A := (A+\Id)^{-1}$ which turns out 
to be a nice firmly nonexpansive operator with full domain.
Resolvents not only provide an alternative view on monotone operators 
because one can recover the underlying maximally monotone operator 
via $J_{A}^{-1}-\Id$ but they also are crucial for the formulation 
of algorithms for finding zeros of $A$ (e.g., the celebrated 
proximal point algorithm \cite{Rockprox}). 

The main purpose of this section is to derive a 
formula for the resolvent $J_{\gamma M} := (\Id+\gamma M)^{-1}$, 
where 
\begin{empheq}[box=\fbox]{equation*}
  \gamma > 0. 
\end{empheq}
Because $\gamma M$
is still maximally monotone, 
the resolvent $J_{\gamma M}$ has full domain. 
Let us start with a result that holds true for displacement mappings 
of linear nonexpansive mappings.

\begin{theorem}
\label{t:genres}
Let $S\colon X\to X$ be nonexpansive and linear.
Then 
\begin{equation*}
  J_{\gamma(\Id-S)} = 
  \sum_{k=0}^\infty \frac{\gamma^k}{(1+\gamma)^{k+1}}S^k.
\end{equation*}
\end{theorem}
\begin{proof}
We have, using \cite[Theorem~7.3-1]{Kreyszig} in \cref{e:201001d}, 
\begin{align}
J_{\gamma(\Id-S)}
&= \big(\Id+\gamma(\Id-S)\big)^{-1}
= \Big((1+\gamma)\big(\Id - \tfrac{\gamma}{1+\gamma}S\big)\Big)^{-1}\notag\\
&= \big(\Id-\tfrac{\gamma}{1+\gamma}S\big)^{-1}\circ \tfrac{1}{1+\gamma}\Id 
= \tfrac{1}{1+\gamma}\big(\Id-\tfrac{\gamma}{1+\gamma}S\big)^{-1}\notag\\
&= \tfrac{1}{1+\gamma}\sum_{k=0}^{\infty}\big(\tfrac{\gamma}{1+\gamma})^kS^k \label{e:201001d}
\end{align}
because $\|\gamma/(1+\gamma)S\| \leq \gamma/(1+\gamma)<1$. 
\end{proof}

\begin{corollary} {\rm\bf (resolvent of $\gamma M$)}
\label{l:Julian}
We have 
\begin{equation}
\label{e:201001c}
J_{\gamma M} = \frac{1}{(1+\gamma)^m-\gamma^m}
\sum_{k=0}^{m-1}(1+\gamma)^{m-1-k}\gamma^kR^k.
\end{equation}
\end{corollary}
\begin{proof}
Using \cref{t:genres} and the assumption that $R$ is of finite order $m$, 
we have 
\begin{align*}
J_{\gamma M}
&= \frac{1}{1+\gamma}\sum_{k=0}^\infty \Big(\frac{\gamma}{1+\gamma}\Big)^kR^k \\
&= \frac{1}{1+\gamma}\bigg(\sum_{k=0}^{m-1} \Big(\frac{\gamma}{1+\gamma}\Big)^kR^k 
+ \Big(\frac{\gamma}{1+\gamma}\Big)^m\sum_{k=0}^{m-1} \Big(\frac{\gamma}{1+\gamma}\Big)^kR^k + \cdots \bigg)\\
&= \frac{1}{1+\gamma} \bigg(1 + \Big(\frac{\gamma}{1+\gamma}\Big)^m + \Big(\frac{\gamma}{1+\gamma}\Big)^{2m}+\cdots \bigg)
\sum_{k=0}^{m-1} \Big(\frac{\gamma}{1+\gamma}\Big)^kR^k\\
&= \frac{1}{1+\gamma} \frac{1}{\displaystyle 1 - \Big(\frac{\gamma}{1+\gamma}\Big)^m}
\sum_{k=0}^{m-1} \Big(\frac{\gamma}{1+\gamma}\Big)^kR^k\\
&=\frac{(1+\gamma)^{m-1}}{(1+\gamma)^m-\gamma^m}
\sum_{k=0}^{m-1} \Big(\frac{\gamma}{1+\gamma}\Big)^kR^k
\end{align*}
and the result follows. 
\end{proof}

\begin{remark}
Consider the formula for $J_{\gamma M}$ from \cref{l:Julian}. 
The coefficients for $R^k$ are positive and sum up to $1$; hence, 
\begin{equation*}
J_{\gamma M}\in\conv\{\Id,R,\ldots,R^{m-1}\}
\end{equation*}
and $J_{\gamma M}|_{\ker M}=\Id|_{\Fix R}$.
In particular, if $\Fix R = \ker M \supsetneqq\{0\}$, 
then $J_{\gamma M}$ cannot be a Banach contraction. 
\end{remark}

Next, we set 
\begin{empheq}[box=\fbox]{equation*}
D := \ker(M) = \Fix R,
\end{empheq}
which is a closed linear subspace of $X$. 
This allows us to describe the asymptotic behaviour of $J_{\gamma M}$
as $\gamma$ tends either to $0^+$ or to $\pinf$. 

\begin{proposition}
\label{p:primalasymp}
We have 
\begin{equation}
\label{e:primalasymp0}
  \lim_{\gamma\to 0^+} J_{\gamma M} = \Id
\end{equation}
and 
\begin{equation}
\label{e:primalasympinf}
  \lim_{\gamma\to \pinf} J_{\gamma M} = P_D = 
  \frac{1}{m}\sum_{k=0}^{m-1}R^k,
\end{equation}
where the limits are understood in the pointwise sense and
$P_D$ denotes the orthogonal projector onto $D$. 
\end{proposition}
\begin{proof}
The main tool is \cref{l:Julian}.

The limit \cref{e:primalasymp0} follows from \cref{e:201001c}, and 
it follows also from 
\cite[Theorem~23.48]{BC2017}
by noting that 
$\dom M = \cdom M = X$ and 
hence $P_{\cdom M} = P_X =\Id$.

Let's turn to \cref{e:primalasympinf}. 
The left equation follows directly from 
\cite[Theorem~23.48(i)]{BC2017} because $D = M^{-1}(0)$. 
Finally, 
as $\gamma\to \pinf$, we have 
\begin{align*}
J_{\gamma M}
&= \frac{1}{\;\;\displaystyle \frac{(1+\gamma)^m-\gamma^m}{\gamma^m}\;\;}
\sum_{k=0}^{m-1}\frac{(1+\gamma)^{m-1-k}\gamma^k}{\gamma^m}R^k\\
&=\frac{1}{\big(1+\frac{1}{\gamma}\big)^m - 1^m}
\,\frac{1}{\gamma}\,\sum_{k=0}^{m-1}\big(1+\tfrac{1}{\gamma}\big)^{m-1-k}R^k\\
&=\frac{1}{\;\;\displaystyle \frac{\big(1+\frac{1}{\gamma}\big)^m - 1^m}{\frac{1}{\gamma}}\;\;}
\sum_{k=0}^{m-1}\big(1+\tfrac{1}{\gamma}\big)^{m-1-k}R^k\\
&\to \frac{1}{m}\sum_{k=0}^{m-1}R^k
\end{align*}
because the derivative of $\xi\mapsto \xi^m$ at $1$ is $m$. 
\end{proof}

\begin{remark}
\label{r:commute}
If both $T_1$ and $T_2$ are operators that are polynomials in $R$, then 
clearly $T_1$ and $T_2$ commute: 
\begin{equation*}
T_1T_2 = T_2T_1. 
\end{equation*}
In particular, by \cref{p:primalasymp}, 
both $P_D$ and $P_{D^\perp} = \Id - P_D$, the projector
onto $D$ and its orthogonal complement respectively, 
commute with any operator that is a polynomial in $R$. 
\end{remark}

We conclude this section with a connection to the mean ergodic theorem. 

\begin{remark}
The linear mean ergodic theorem (see, e.g.,
\cite[Theorem~II.11]{ReSi} 
and \cite[Chapter~X, Section~144]{RSz})
states that 
\begin{equation}
\label{e:201001b}
  P_{\Fix S} = \lim_{n\to\infty}\frac{1}{n}\sum_{k=0}^{n-1} S^k
\end{equation}
pointwise for any surjective isometric (or even just nonexpansive) linear operator $S 
\colon X\to X$. 
Because $R^m=\Id$ and $D=\Fix R$,  
\cref{e:201001b} yields
in particular as $N\to\pinf$ that 
\begin{equation*}
  P_D \leftarrow
  \frac{1}{Nm-1}\sum_{k=0}^{Nm-1}R^k
  = 
  \frac{N}{Nm-1}\sum_{k=0}^{m-1}R^k
  \rightarrow
  \frac{1}{m}\sum_{k=0}^{m-1}R^k, 
\end{equation*}
i.e., an alternative proof of the right identity in \cref{e:primalasympinf}.
\end{remark}

\section{The inverse $M^{-1}$ and Moore-Penrose inverse $M^\dagger$} 
\label{sec:inverse}

The operator $M$ is a continuous linear operator on $X$. 
Unfortunately, $M$ is neither injective nor surjective. 
This raises some very natural questions: 
What is the set-valued inverse $M^{-1}$? What is the Moore-Penrose inverse of $M$?
It is very satisfying that complete answers to these questions are possible,
and we provide these in this section. 
We start by considering the kernel and the range of $M$.

\begin{proposition}
We have 
\begin{equation}
\label{e:200523a}
\ker M = D = \ker M^*
\end{equation}
and 
\begin{equation}
\label{e:200523b}
\ran M = D^\perp = \ran M^*;
\end{equation}
in particular, $\ran M$ is closed. 
\end{proposition}
\begin{proof}
By  \cite[Proposition~20.17]{BC2017},
\begin{equation}
\label{e:200523c}
\ker M = \ker M^* \text{~and~}  \cran M = \cran M^*.
\end{equation}
This and the definition of $D$ yield \cref{e:200523a}.

Now let $y\in X$.
Assume first that $y\in\ran M$.
Then $y\in \cran M = \cran M^* = (\ker M)^\perp = D^\perp$.
Conversely, we assume that 
$y\in D^\perp$ and we set 
\begin{equation*}
  x := \frac{1}{m}\sum_{k=0}^{m-2}(m-1-k)R^{k}y.
\end{equation*}
Using \cref{e:primalasympinf} in \cref{e:201001a}, we obtain 
\begin{align}
Mx 
&= (\Id-R)x\notag \\
&= \frac{1}{m}\sum_{k=0}^{m-2}(m-1-k)R^{k}y 
- \frac{1}{m}\sum_{k=0}^{m-2}(m-1-k)R^{k+1}y\notag \\
&= \frac{1}{m}\sum_{k=0}^{m-2}(m-1-k)R^{k}y 
- \frac{1}{m}\sum_{k=1}^{m-1}(m-k)R^{k}y\notag \\
&= \frac{m-1}{m}y -\frac{1}{m}R^{m-1}y
- \frac{1}{m}\sum_{k=1}^{m-2}R^k y\notag\\
&=
\Big(\Id -\frac{1}{m}\sum_{k=0}^{m-1}R^k \Big)y\notag\\
&= \big( \Id- P_D\big)y \label{e:201001a}\\
&= P_{D^\perp} y\notag \\
&= y.\notag
\end{align}
Hence $y=Mx \in \ran M$ and thus $D^\perp \subseteq \ran M$.
Altogether, we see that 
\begin{equation}
\label{e:200523d}
\ran M = D^\perp \text{~is closed.}
\end{equation}
The remaining conclusion follows by combining 
\cref{e:200523d}, \cref{e:200523c}, and the fact that 
$\ran M^*$ is closed (because $\ran M$ is and 
\cite[Corollary~15.34]{BC2017} applies).
\end{proof}

We now define the continuous linear operator 
\begin{empheq}[box=\fbox]{equation}
\label{e:defT}
T := \frac{1}{2m}\sum_{k=1}^{m-1} (m-2k)R^k
  \end{empheq}
which will turn out to be key to the study of $M^{-1}$. 

\begin{proposition}
\label{p:T}
The operator $T$ satisfies the following:
\begin{enumerate}
  \item 
  \label{p:T:floor}
  \begin{equation}
    \label{e:200520a}
  T = \frac{1}{2m}\sum_{k=1}^{\lfloor m/2 \rfloor}
  (m-2k)\big(R^k - R^{m-k}\big).
  \end{equation}
  \item 
  \label{p:T:skew} 
  $T^*=-T$, i.e., $T$ is a \emph{skew} linear operator. 
  \item 
  \label{p:T:ran} 
  $\ran T \subseteq D^\perp$. 
\end{enumerate}
\end{proposition}
\begin{proof}
\cref{p:T:floor}: 
This follows easily by considering two cases ($m$ is odd and $m$ is even).

\cref{p:T:skew}:
The skew part of $R^k$ 
is $\thalb(R^k-(R^k)^*)
=\thalb(R^k-(R^{*})^k)
=\thalb(R^k-(R^{-1})^k)
=\thalb(R^k-R^{m-k})$. 
Hence each term in the sum \eqref{e:200520a} is skew, and therefore so is $T$.

\cref{p:T:ran}:
The formula for $P_D$ in \cref{e:primalasympinf} 
yields 
\begin{equation*}
  P_{D^\perp} = \frac{m-1}{m}\Id - \frac{1}{m}\sum_{i=1}^{m-1}R^i.
\end{equation*}
Using this and \cref{e:defT}, we obtain (using the empty-sum convention)
\begin{align*}
  &\qquad 
2m^2P_{D^\perp}T
= 
\bigg((m-1)\Id - \sum_{i=1}^{m-1}R^i \bigg)
\bigg(\sum_{j=1}^{m-1}(m-2j)R^j \bigg)\\
&= \Big( - \sum_{i=1}^{m-1}(m-2(m-i)) \Big)\Id \\
&\qquad + \sum_{k=1}^{m-1}\bigg((m-1)(m-2k) -\sum_{i=1}^{k-1}(m-2(k-i))- 
\sum_{i=k+1}^{m-1}(m-2(m+k-i)) \bigg)R^k\\
&= (0)\Id + \sum_{k=1}^{m-1}\Big( (m-1)(m-2k)
-(k-1)(m-k) + k(m-1-k)
\Big)R^k\\
&= \sum_{k=1}^{m-1} m(m-2k)R^k
=2m^2 T.
\end{align*}
Hence $P_{D^\perp}T = T$; equivalently, $\ran T \subseteq D^\perp$. 
\end{proof}

We are now able to provide a formula for the inverse of $M$.

\begin{theorem}
\label{t:Minv}
We have 
\begin{equation*}
  M^{-1}= \thalb\Id + T + N_{D^\perp},
\end{equation*}
where $N_{D^\perp} = \partial \iota_{D^\perp}$ 
denotes the normal cone operator of $D^\perp$. 
\end{theorem}
\begin{proof}
Using \cref{e:200523b}, we observe that 
$\dom(T+N_{D^\perp}) = D^\perp = \ran M
= \dom M^{-1} = \dom (M^{-1}-\thalb\Id)$. 
So pick an arbitrary 
\begin{equation*}
y\in D^\perp.
\end{equation*}
In view of the definition of $N_{D^\perp}$, it suffices to show that 
\begin{equation*}
  M^{-1}y - \thalb y \stackrel{?}{=} Ty + D.
\end{equation*}
Let $x\in M^{-1}y$. Then $Mx=y$ and $M^{-1}y = x + \ker M = x +D$ 
by \cref{e:200523a}. 
Hence we must show that 
\begin{equation*}
  x+D - \thalb y \stackrel{?}{=} Ty + D,
\end{equation*}
which is equivalent to 
$  x+D - \thalb (x-Rx) \stackrel{?}{=} T(x-Rx) + D$ 
and to 
\begin{equation*}
  (x+Rx) + D  \stackrel{?}{=} 2T(x-Rx) + D.
\end{equation*}
Note that 
$P_{D^\perp}(x+Rx) = P_{D^\perp}(2T(x-Rx))$
$\Leftrightarrow$ 
$2T(x-Rx)-(x+Rx)\in D=\ker M$
$\Leftrightarrow$ 
$M(2T(x-Rx)) = M(x+Rx)=(\Id-R)(x+Rx)$.
Hence we must prove that 
\begin{equation}
\label{e:200523e}
  M\big(2T(x-Rx)\big) \stackrel{?}{=} x - R^2x.
\end{equation}
We now work toward the proof of \cref{e:200523e}.
First, note that  
\begin{align*}
2mT(x-Rx)
&= 
\sum_{k=1}^{m-1}(m-2k)R^k(x-Rx)\\
&= 
\sum_{k=1}^{m-1}(m-2k)R^kx
-\sum_{k=1}^{m-1}(m-2k)R^{k+1}x\\
&= 
\sum_{k=1}^{m-1}(m-2k)R^kx
-\sum_{k=2}^{m}(m+2-2k)R^{k}x\\
&= 
(m-2)x + (m-2)Rx -2\sum_{k=2}^{m-1}R^kx,
\end{align*}
which implies
\begin{equation}
\label{e:200523f}
  2T(x-Rx)
  = \frac{m-2}{m}x + \frac{m-2}{m}Rx 
  - \frac{2}{m}\sum_{k=2}^{m-1}R^kx.
\end{equation}
Using \cref{e:200523f}, we see that 
\begin{align*}
M\big(2T(x-Rx)\big)
&=(\Id-R)\big(2T(x-Rx)\big) \\
&=\big(2T(x-Rx) \big) - R\big(2T(x-Rx) \big)\\
&= \frac{m-2}{m}x + \frac{m-2}{m}Rx 
- \frac{2}{m}\sum_{k=2}^{m-1}R^kx \\
&\quad -R\bigg(\frac{m-2}{m}x + \frac{m-2}{m}Rx 
- \frac{2}{m}\sum_{k=2}^{m-1}R^kx\bigg)\\
&= \frac{m-2}{m}x - \frac{2}{m}\sum_{k=2}^{m-1}R^kx 
-\frac{m-2}{m}R^2x + \frac{2}{m}\sum_{k=2}^{m-1}R^{k+1}x\\
&= \frac{m-2}{m}x - \frac{2}{m}R^2x -\frac{m-2}{m}R^2x + \frac{2}{m}x\\
&=x - R^2x, 
\end{align*}
i.e., \cref{e:200523e} does hold, 
as desired!
\end{proof}

\begin{remark}
\label{r:skewsharp}
Consider \cref{t:Minv}.
Then $M^{-1}-\thalb\Id$ is monotone; in other words, 
$M^{-1}$ is \emph{$\thalb$-strongly monotone}.
If $R\neq\Id$, then 
$D\neq X$; hence 
$D^\perp\supsetneqq \{0\}$
and $T|_{D^\perp}-\varepsilon\Id$ cannot 
be monotone because $T$ is skew. 
It follows that the constant $\thalb$ is sharp:
\begin{equation*}
\text{
$M^{-1}$ is not $\sigma$-strongly monotone 
if $R\neq\Id$ and $\sigma > \thalb$.
}
\end{equation*}
\end{remark}

Recall that given $y\in X$, the vector $M^\dagger y$
is the (unique) minimum norm vector among all the solution vectors $x$ 
to the least squares problem
\begin{equation*}
\|Mx-y\| = \min_{z\in X}\|Mz-y\|. 
\end{equation*}
(We refer the reader to \cite{Groetsch} for further information on generalized inverses.)
We now present a pleasant formula for 
the Moore-Penrose inverse $M^\dagger$ of $M$.

\begin{theorem}
The Moore-Penrose inverse of $M$ is
\begin{equation}
\label{e:250523h}
 M^\dagger = \thalb P_{D^\perp}(\Id+2T) = \sum_{k=0}^{m-1} \frac{m-1-2k}{2m}R^{k}.
\end{equation}
\end{theorem}
\begin{proof}
Recall that by \cref{e:primalasympinf}
\begin{equation*}
  \Id-P_D = P_{D^\perp} = \Id - \frac{1}{m}\sum_{k=0}^{m-1}R^k,
\end{equation*}
which is a polynomial in $R$. 
Using also \cite[Proposition~2.1]{Victoria} and 
the fact that $T$ is also a polynomial in $R$ (see Remark~\ref{r:commute}), we have 
\begin{align}
M^\dagger &= P_{\ran M^*} \circ M^{-1} \circ P_{\ran M}
= P_{D^\perp} \circ ( \thalb\Id + T + N_{D^\perp})\circ P_{D^\perp}\notag\\
&= \thalb P_{D^\perp} + P_{D^\perp}TP_{D^\perp}\notag\\
&= \thalb P_{D^\perp}(\Id+2T)\notag\\
&=\frac{1}{2m} \Big((m-1)\Id - \sum_{i=1}^{m-1}R^i\Big)\big(\Id+2T\big)\notag\\
&=\frac{1}{2m} \Big((m-1)\Id - \sum_{i=1}^{m-1}R^i\Big)
\Big(\Id + \frac{1}{m}\sum_{j=1}^{m-1}(m-2j)R^j \Big)\notag\\
&=\frac{1}{2m^2} \Big((m-1)\Id - \sum_{i=1}^{m-1}R^i\Big)
\Big(m\Id + \sum_{j=1}^{m-1}(m-2j)R^j \Big)\notag\\
&=\frac{1}{2m^2} \Big((m-1)\Id - \sum_{i=1}^{m-1}R^i\Big)
\Big(\sum_{j=0}^{m-1}(m-2j)R^j \Big). \label{e:250523g}
\end{align}
We thus established the left identity in \cref{e:250523h}.
To obtain the right identity in \cref{e:250523h}, 
we use the very last expression \cref{e:250523g} for $M^\dagger$ and compute 
the coefficients of $\Id,R,R^2,\ldots,R^{m-1}$. 

The coefficient for $\Id$ is 
\begin{equation*}
  \frac{1}{2m^2}\Big((m-1)m - \sum_{i=1}^{m-1}(m-2(m-i)) \Big)
  = \frac{m-1}{2m}
\end{equation*}
as needed.
The coefficient for $R$ is 
\begin{equation*}
  \frac{1}{2m^2}\Big((m-1)(m-2)-(m) -\sum_{i=2}^{m-1}(m-2(m+1-i)) \Big)
  = \frac{m-3}{2m}
\end{equation*}
as needed.
The coefficient for the general $R^k$ is 
\begin{equation*}
  \frac{\displaystyle (m-1)(m-2i)-\sum_{i=1}^{k}(m-2(k-i)) 
  - \sum_{i=k+1}^{m-1}(m-2(m+k-i))}{2m^2}
  = \frac{m-1-2k}{2m}
\end{equation*}
as needed. We thus verified the right equality in \cref{e:250523h}. 
\end{proof}

\begin{corollary}
For all $y\in\ran M = D^\perp$, we have 
\begin{equation*}
  M^{-1}y = M^\dagger y + D = 
  D +  \sum_{k=0}^{m-1} \frac{m-1-2k}{2m}R^{k}y. 
\end{equation*}
\end{corollary}
\begin{proof}
It is well known (see, e.g., \cite[Proposition~3.31]{BC2017}) that 
$MM^\dagger = P_{\ran M}$ which readily implies the conclusion. 
\end{proof}

\begin{remark} 
We mention in passing that the results in this section combined with work 
on decompositions of monotone linear relations lead to a
Borwein-Wiersma decomposition
\begin{equation*}
 M^{-1} = \partial\big( \tfrac{1}{4}\|\cdot\|^2 + \iota_{D^\perp}\big) + T
\end{equation*}
of $M^{-1}$. The required background is nicely detailed in 
Liangjin Yao's PhD thesis \cite{Liangjin}. 
\end{remark}

\section{Resolvents and Yosida approximations}

\label{sec:ResYo}

Given a maximally monotone operator $A$ on $X$,
recall that its resolvent is defined by $J_A =
(\Id+A)^{-1}$. 
We have computed the resolvent $J_{\gamma M}$ 
already in \cref{l:Julian}. In this section, 
we present a formula for $J_{\gamma M^{-1}}$.
Moreover, we provide formulas for 
the \emph{Yosida approximations} $\moyo{M}{\gamma}$
and $\moyo{M^{-1}}{\gamma}$, and we also discuss Lipschitz properties 
of $J_{\gamma M}$. 
Recall that the Yosida approximation of $A$ of index $\gamma>0$
is defined by $\moyo{A}{\gamma} := \tfrac{1}{\gamma}(\Id-J_{\gamma A})$.
Yosida approximations are powerful tools to study monotone
operators. They can be viewed as regularizations and 
approximations of $A$ because  $\moyo{A}{\gamma}$ is a single-valued 
Lipschitz continuous operator on $X$ and $\moyo{A}{\gamma}$ approximates 
$A$ in the sense that 
$\moyo{A}{\gamma}x\to P_{Ax}(0)\in Ax$ as $\gamma\to 0^+$. 
(For this and more, see, e.g., \cite[Chapter~23]{BC2017}.) 

Recall that we proved in \cref{l:Julian} that 
\begin{equation}
\label{e:Julianagain}
J_{\gamma M} = \frac{1}{(1+\gamma)^m-\gamma^m}
\sum_{k=0}^{m-1}(1+\gamma)^{m-1-k}\gamma^kR^k;
\end{equation}
indeed, this will give us the following result quickly.

\begin{theorem}
We have 
\begin{equation}
\label{e:Jinv}
  J_{\gamma M^{-1}} = \Id-J_{(1/\gamma)M},
\end{equation}
\begin{equation}
\label{e:moyo}
  \moyo{M}{\gamma}= \tfrac{1}{\gamma}\big(\Id-J_{\gamma M} \big),
\end{equation}
and 
\begin{equation}
\label{e:moyoinv}
  \moyo{(M^{-1})}{\gamma}= \tfrac{1}{\gamma}J_{(1/\gamma)M}
  = \frac{1}{(1+\gamma)^m-1}
  \sum_{k=0}^{m-1}(1+\gamma)^{m-1-k}R^k.
\end{equation}
\end{theorem}
\begin{proof}
By the linearity of $M$ and 
\cite[Proposition~23.20]{BC2017},
we have 
\begin{equation*}
J_{\gamma M^{-1}}
= 
({1}/{\gamma})
J_{(1/\gamma)^{-1}M^{-1}}\circ (1/\gamma)^{-1}\Id
= \Id - J_{(1/\gamma)M},
\end{equation*}
and \cref{e:Jinv} holds. 
\cref{e:moyo} is just the definition, 
while \cref{e:moyoinv} is clear from \cref{e:Jinv}.
The remaining identity follows by employing \cref{e:Julianagain}.
\end{proof}

We already observed in $M^{-1}$ is $\thalb$-strongly monotone and 
that the constant $\thalb$ is sharp (unless $R=\Id$). 
This implies that the resolvent $J_{\gamma M^{-1}}$ is a 
Banach contraction --- in general, this observation is already 
in Rockafellar's seminal paper \cite{Rockprox}:

\begin{proposition}
\label{p:Banach}
$J_{\gamma M^{-1}}$ is a Banach contraction 
with Lipschitz constant $2/(2+\gamma)<1$.
If $D\neq\{0\}$, then $J_{\gamma M}$ cannot be a Banach contraction. 
\end{proposition}
\begin{proof}
Using \cref{r:skewsharp} and \cite[Proposition~23.13]{BC2017}, 
we have the implications 
$M^{-1} - \thalb\Id$ is skew (and hence monotone) 
$\Leftrightarrow$
$\gamma M^{-1}-(\gamma/2)\Id$ is monotone
$\Leftrightarrow$
$\gamma M^{-1}$ is $\beta$-strongly monotone with $\beta := \gamma/2$
$\Leftrightarrow$
$J_{\gamma M^{-1}}$ is $(1+\beta)$ cocoercive 
$\Rightarrow$ 
$J_{\gamma M^{-1}}$ is a Banach contraction with constant 
$1/(1+\beta) = 2/(2+\gamma)< 1$.

Finally, if $D\neq\{0\}$, then 
$D=\ker (\gamma M) = \Fix J_{\gamma M}$ contains infinitely many points and 
thus $J_{\gamma M}$ cannot be a Banach contraction. 
\end{proof}

We will see in \cref{subsec:rot} 
below that the contraction constant cannot be improved in general.

\section{Examples}

\label{sec:examples}

In this section, we turn to concrete examples to illustrate our results. 
It is highly satisfying that complete formulas are available.
For reasons of space, we restrict our attention to $m\in\{2,3,4\}$. 
The underlying linear isometry of finite order will 
be either a rational rotator or a circular right-shift operator.


\subsection{Rational rotators}
\label{subsec:rot}

Let us start with rational rotators. 
In this subsection, we assume that 
\begin{equation*}
  X = \RR^2
  \text{~and~}
  R = \begin{pmatrix}
  \cos(2\pi/m) & - \sin(2\pi/m)  \\
  \sin(2\pi/m) & \cos(2\pi/m)  
  \end{pmatrix}.
\end{equation*}
Then 
$R^m=\Id$ and $D = \Fix R = \{(0,0)\}\subseteq X$.
Using \cref{e:Julianagain} and \cref{e:Jinv}, 
we record the following formulas for 
$J_{\gamma M}$, $J_{\gamma M^{-1}}$, 
$\moyo{M}{\gamma}$, and 
$\moyo{(M^{-1})}{\gamma}$:\\
If $m=2$, then 
\begin{equation*}
  J_{\gamma M} =
  \frac{1}{1+2\gamma}
  \begin{pmatrix}
  1 & 0 \\
  0 & 1 
  \end{pmatrix},
  \;\;
  J_{\gamma M^{-1}} =
  \frac{2}{2+\gamma}
  \begin{pmatrix}
  1 & 0 \\
  0 & 1
  \end{pmatrix},
\end{equation*}
\begin{equation*}
  \moyo{M}{\gamma} =
  \frac{2}{1+2\gamma}
  \begin{pmatrix}
  1 & 0 \\
  0 & 1 
  \end{pmatrix}, 
  \;\;\text{and}\;\;
  \moyo{(M^{-1})}{\gamma} =
  \frac{1}{2+\gamma}
  \begin{pmatrix}
  1 & 0 \\
  0 & 1
  \end{pmatrix}.
\end{equation*}
If $m=3$, then 
\begin{equation*}
  J_{\gamma M} =
  \frac{1}{2+6\gamma+6\gamma^2}
  \begin{pmatrix}
  2+3\gamma & -\sqrt{3}\gamma\\
  \sqrt{3}\gamma &  2+3\gamma
  \end{pmatrix}, 
  \;\;
  J_{\gamma M^{-1}} =
  \frac{1}{6+6\gamma+2\gamma^2}
  \begin{pmatrix}
  6 + 3\gamma& \sqrt{3}\gamma \\
  -\sqrt{3}\gamma & 6+3\gamma 
  \end{pmatrix},
\end{equation*}
\begin{equation*}
  \moyo{M}{\gamma} =
  \frac{1}{2+6\gamma+6\gamma^2}
  \begin{pmatrix}
  3+6\gamma & \sqrt{3}\\
  -\sqrt{3} &  3+6\gamma
  \end{pmatrix},
  \;\;\text{and}\;\;
  \moyo{(M^{-1})}{\gamma} =
  \frac{1}{6+6\gamma+2\gamma^2}
  \begin{pmatrix}
  3 + 2\gamma& -\sqrt{3} \\
  \sqrt{3} & 3+2\gamma 
  \end{pmatrix}.
\end{equation*}
If $m=4$, then 
\begin{equation*}
  J_{\gamma M} =
  \frac{1}{1+2\gamma+2\gamma^2}
  \begin{pmatrix}
  1+\gamma & -\gamma\\
  \gamma &  1+\gamma
  \end{pmatrix},
  \;\;
  J_{\gamma M^{-1}} =
  \frac{1}{2+2\gamma+\gamma^2}
  \begin{pmatrix}
  2+\gamma & \gamma \\
  -\gamma & 2+\gamma 
  \end{pmatrix},
\end{equation*}
\begin{equation*}
  \moyo{M}{\gamma} =
  \frac{1}{1+2\gamma+2\gamma^2}
  \begin{pmatrix}
  1+2\gamma & 1\\
  -1 &  1+2\gamma
  \end{pmatrix}, 
  \;\;\text{and}\;\;
  \moyo{(M^{-1})}{\gamma} =
  \frac{1}{2+2\gamma+\gamma^2}
  \begin{pmatrix}
  1+\gamma & -1\\
  1 & 1+\gamma 
  \end{pmatrix}.
\end{equation*}

Higher values of $m$ lead to unwieldy matrices. 
We do note that for $m=2$, we have $J_{\gamma M^{-1}} = 2/(2+\gamma)\Id$;
consequently, the constant $2/(2+\gamma)$ in \cref{p:Banach} is sharp.

\subsection{Circular shift operators}

In this subsection, we assume that 
$H$ is another real Hilbert space, 
\begin{equation*}
  X = H^m \text{~and~} 
  R \colon X \to X\colon (x_1,x_2,\ldots,x_m)\mapsto (x_m,x_1,\ldots,x_{m-1})
\end{equation*}
is the circular right-shift operator. 
Then 
\begin{equation*}
  R^m = \Id \text{~and~}
  D = \Fix R = \menge{(x,x,\ldots,x)\in X}{x\in H},
\end{equation*}
is the ``diagonal'' subspace of $X$
with orthogonal complement 
\begin{equation*}
D^\perp = \menge{(x_1,x_2,\ldots,x_m)\in X}{x_1+x_2+\cdots +x_m=0}.
\end{equation*}
Some of our results in this paper were derived in \cite{Victoria};
however, with less elegant proofs (sometimes relying on the specific form of $R$).
Using \cref{e:Julianagain} and \cref{e:Jinv} once again, 
we calculate 
$J_{\gamma M}$, $J_{\gamma M^{-1}}$, 
$\moyo{M}{\gamma}$, and 
$\moyo{(M^{-1})}{\gamma}$
for $m\in\{2,3\}$.
These formulas are all new. 
The matrices in the following are to 
be interpreted as block matrices where a numerical 
entry $\alpha$ stands for $\alpha\Id|_H$. 
If $m=2$, then 
\begin{equation*}
  J_{\gamma M} =
  \frac{1}{1+2\gamma}
  \begin{pmatrix}
  1+\gamma & \gamma \\
  \gamma & 1+\gamma 
  \end{pmatrix}
  \;\;\text{and}\;\;
  J_{\gamma M^{-1}} =
  \frac{1}{2+\gamma}
  \begin{pmatrix}
  1 & -1 \\
  -1 & 1
  \end{pmatrix};
\end{equation*}
and 
\begin{equation*}
  \moyo{M}{\gamma} =
  \frac{1}{1+2\gamma}
  \begin{pmatrix}
  1 & -1 \\
  -1 & 1 
  \end{pmatrix}
  \;\;\text{and}\;\;
  \moyo{(M^{-1})}{\gamma} =
  \frac{1}{(2+\gamma)\gamma}
  \begin{pmatrix}
  1+\gamma & 1 \\
  1 & 1+\gamma
  \end{pmatrix}.
\end{equation*}
If $m=3$, then 
\begin{equation*}
  J_{\gamma M} =
  \frac{1}{1+3\gamma+3\gamma^2}
  \begin{pmatrix}
  (1+\gamma)^2 & \gamma^2 & (1+\gamma)\gamma\\
  (1+\gamma)\gamma & (1+\gamma)^2 & \gamma^2\\
  \gamma^2 & (1+\gamma)\gamma & (1+\gamma)^2
  \end{pmatrix};
\end{equation*}
\begin{equation*}
  J_{\gamma M^{-1}} =
  \frac{1}{3+3\gamma+\gamma^2}
  \begin{pmatrix}
  2+\gamma & -1 & -(1+\gamma)\\
  -(1+\gamma) & 2+\gamma & -1\\
  -1 & -(1+\gamma) & 2+\gamma
  \end{pmatrix};
\end{equation*}
\begin{equation*}
  \moyo{M}{\gamma} =
  \frac{1}{1+3\gamma+3\gamma^2}
  \begin{pmatrix}
  1+2\gamma & -\gamma & -1-\gamma\\
  -1-\gamma & 1+2\gamma & -\gamma\\
  -\gamma & -1-\gamma & 1+2\gamma
  \end{pmatrix};
\end{equation*}
\begin{equation*}
  \moyo{(M^{-1})}{\gamma} =
  \frac{1}{(3+3\gamma+\gamma^2)\gamma}
  \begin{pmatrix}
  (1+\gamma)^2 & 1 & 1+\gamma\\
  1+\gamma & (1+\gamma)^2 & 1\\
  1 & 1+\gamma & (1+\gamma)^2
  \end{pmatrix}.
\end{equation*}
We refrain from listing formulas for $m\geq 4$ due to their complexity. 

\section{Concluding remarks}

\label{sec:conc}

Monotone operator theory is a fascinating and useful area 
of set-valued and variational analysis. 
In this paper, we carefully analyzed the displacement mapping
of an isometry of finite order. 
We provided new and explicit formulas for inverses 
(both in the set-valued and Moore-Penrose sense) as well as for 
resolvents and Yosida approximants. 
This is a valuable contribution because explicit formulas 
are rather uncommon in this area. 
We believe these formulas will turn out to be useful 
not only for discovering new results but also for providing examples or counterexamples. 
(See also \cite{settled} and \cite{ABW21} for very recent applications 
of our work.)

\section*{Acknowledgments}
HHB and XW are supported by the Natural Sciences and
Engineering Research Council of Canada.

\end{document}